\documentclass[11pt]{article}
\usepackage{amsmath}
\usepackage{amssymb,latexsym, amsmath, amscd, array, graphicx,relsize}
\usepackage[all]{xy}

\newtheorem{theorem}{Theorem}
\newtheorem{lemma}[theorem]{Lemma}
\newtheorem{proposition}[theorem]{Proposition}
\newtheorem{corollary}[theorem]{Corollary}

\newtheorem{prob}[theorem]{Problem}
\newtheorem{rem}[theorem]{Remark}

\newenvironment{proof}[1][Proof]{\begin{trivlist}
\item[\hskip \labelsep {\bfseries #1}]}{\end{trivlist}}
\newenvironment{definition}[1][Definition]{\begin{trivlist}
\item[\hskip \labelsep {\bfseries #1}]}{\end{trivlist}}

\newcommand{\qed}{\nobreak \ifvmode \relax \else
      \ifdim\lastskip<1.5em \hskip-\lastskip
      \hskip1.5em plus0em minus0.5em \fi \nobreak
      \vrule height0.75em width0.5em depth0.25em\fi}

\title{On separators of the space of complete non-negatively curved metrics on the plane}

\date{}
\begin{document}

\maketitle

\begin{abstract}We shall prove that the Hilbert cube cannot be separated by  a weakly infinite dimensional subset. As a corollary we obtain that the complement of a weakly infinite dimensional subset of the space of complete non negatively curved metrics is  continuum connected. We can extend this result to the associated moduli space when the set removed is a Hausdorff space with Haver's property $\mathcal{C}$. These results are refinements of  theorems proven by  Belegradek and Hu~\cite{BH}.
\end{abstract}

\section{Introduction}

The spaces of Riemannian metrics with positive scalar curvature are subjects of intensive study~\cite{Ro}.
The  connectedness properties of such spaces on $\mathbb R^2$ were studied recently by Belegradek and Hu in~\cite{BH}.
They proved that in the space $\mathcal{R}^{k}_{\ge 0}(\mathbb{R}^2)$  of complete Riemannian metrics of non negative curvature on the plane equipped with the  topology of $C^k$ uniform convergence on compact sets, the complement $\mathcal{R}^{k}_{\ge 0}(\mathbb{R}^2)\setminus X$ is connected for every finite dimensional $X$. 
We note that the space $\mathcal{R}^{k}_{\ge 0}(\mathbb{R}^2)$ is separable metric~\cite{BH}.
In this note we extend Belegradek-Hu's result to the case of infinite dimensional spaces $X$.
We recall that infinite dimensional spaces  split in two disjoint classes: strongly infinite dimensional (like the Hilbert cube)
and weakly infinite dimensional (like the union of $\cup_n I^n)$.
We prove Belegradek-Hu's theorem for weakly infinite dimensional $X$. This extension is final since strongly infinite dimensional spaces can separate the Hilbert cube.

We note that in~\cite{BH} there is a similar connectedness result with finite dimensional $X$ for the moduli spaces $\mathcal{M}^{k,c}_{\ge 0}(\mathbb{R}^2)$, i.e. the quotient space of $\mathcal{R}^{k}_{\ge 0}(\mathbb{R}^2)$ by the ${\rm Diff}(\mathbb R^2)$-action via pullback. In the case of moduli spaces we  manage to extend their connectedness result to the subsets $X\subset\mathcal{M}^{k}_{\ge 0}(\mathbb{R}^2)$ with Haver's property $\mathcal{C}$
(called $\mathcal{C}$-spaces in~\cite{En}). It is known that the property $\mathcal{C}$ implies the weak infinite dimensionality~\cite{En}. 
There is an old open problem whether every weakly infinite dimensional compact metric space has property $\mathcal{C}$. For general spaces these two classes are different~\cite{AG}.

\section{Infinite dimensional spaces}

We denote the \textit{Hilbert cube} by $Q=[-1,1]^{\infty}=\Pi_{n=1}^{\infty}\mathbb{I}_n$. The \textit{pseudo interior} of $Q$ is the set $s=(-1,1)^{\infty}$ and the \textit{pseudo boundary} of $Q$ is the set $B(Q)=Q\setminus s$. The \textit{faces} of $Q$ are the sets $W_i^{-}=\{x\in Q|x_i=-1\}$ and  $W_i^{+}=\{x\in Q|x_i=1\}$.  Every space under consideration is a separable metric space. \\
A space $S\subseteq X$ is said to \textit{separate} $X$ if $X\setminus S$ is disconnected. Let $X$ be a space and let $A, B$ be two disjoint closed subsets of $X$, a \textit{separator between $A$ and $B$} is a closed subset $S\subseteq X$ such that $X\setminus S$ can be written as the disjoint union of open sets $U$ and $V$ with $A\subseteq U$ and $B\subseteq V$. 

\begin{definition}
Let $X$ be a space and $\Gamma$  be an index set. A family of  pairs of disjoint closed sets $\tau=\{(A_i,B_i): i\in \Gamma\}$ of $X$ is said to be \textit{essential} if for every family $\{L_i:i\in \Gamma\}$ where $L_i$ is a separator between $A_i$ and $B_i$, we have $\bigcap_{i\in\Gamma}L_i\neq \emptyset$.\\
If $\tau$ is not essential, then it is called \textit{inessential}. 
\end{definition}

We recall that the classical covering dimension can be defined in terms of essential families as follows:
\begin{definition}
For a space $X$ we define, $\textrm{dim} X \in \{-1, 0, 1, \ \cdots\}\cup\{\infty\}$ by\\
\indent $\textrm{dim} X=-1$ iff $X=\emptyset$\\
\indent $\textrm{dim} X\le n$ iff every family of $n+1$ pairs of disjoint closed subsets is\\ \indent\indent\indent\indent inessential.\\
\indent $\textrm{dim} X=n$ iff  $\textrm{dim} X\le n$ and  $\textrm{dim} X\nleq n-1$\\
\indent $\textrm{dim} X=\infty$ iff  $\textrm{dim} X\neq n$ for all $n\ge -1$\\\\
A space $X$ is called \textit{strongly infinite dimensional} if there exists an infinite essential family of pairs of disjoint closed subsets of $X$. $X$ is called \textit{weakly infinite dimensional} if  $X$ is not strongly infinite dimensional.
\end{definition}

We recall that a space $X$ is \emph{continuum connected} if every two points $x, y\in X$ are contained in a connected compact subset.

The following fact is well known. A proof can be found in~\cite{vM}, Corollary 3.7.5. 

\begin{lemma}\label{thm1V} 
Let $X$ be a compact space, let $\{A_i,B_i:i\in\Gamma\}$ be an essential family of pairs of disjoint closed subsets of $X$ and let $n\in\Gamma$. Suppose that $S\subseteq X$ is such that $S$ meets every continuum from $A_n$ to $B_n$. For each $i\in\Gamma\setminus\{n\}$ let $U_i$ and $V_i$ be disjoint closed neighborhoods of $A_i$ and $B_i$ respectively. Then $\{(U_i\cap S, V_i\cap S)\}$ is essential in $S$.
\end{lemma}

\begin{corollary}\label{cor}
Let $S\subseteq Q$ be such that $S$ meets every continuum from $W_1^+$ to $W_1^-$. Then $S$ is strongly infinite dimensional.\qed
\end{corollary}

A set $A\subseteq Q$ is a $\mathcal{Z}$-set in $Q$ if for every open cover $\mathcal{U}$ of $X$ there exists a map of $Q$ into $Q\setminus A$ which is $\mathcal{U}$-close to the identity. It should be noted that any face of $Q$ and any point of $Q$ are $\mathcal{Z}$-sets. The following Theorem is from ~\cite{Ch}, Theorem 25.2.\\

\begin{theorem}\label{chap} Let $A, B\subseteq Q$ be $\mathcal{Z}$-sets such that $\textrm{Sh}(A)=\textrm{Sh}(B)$. Then $Q\setminus A$ is homeomorphic to $Q\setminus B$. 
\end{theorem}

We do not use the notion of shape in full generality. We just recall that for homotopy equivalent spaces $A$ and $B$ we have
$\textrm{Sh}(A)=\textrm{Sh}(B)$.

\begin{theorem}\label{main1} 
Let $x,y\in Q\setminus S$ where $S\subseteq Q$ be such that it intersects every continuum from $x$ to $y$. Then $S$ is strongly infinite dimensional.
\end{theorem}

\begin{proof} Let $x,y\in Q\setminus S$. Applying the Theorem~\ref{chap} to $W_1^-\cup W_1^+$ and $\{x,y\}$ we obtain a homeomorphism $f:Q\setminus (W_1^-\cup W_1^+)\to Q\setminus \{x,y\}$. 
In view of the minimality of the one point compactification, this homeomorphism can be extended
to a continuous map of compactifications
$\bar{f}:Q\to Q$ with $\bar{f}(W_1^-)=x$, $\bar{f}(W_1^+)=y$ and $f|_{Q\setminus (W_1^-\cup W_1^+)}=f$.  
Note that $\bar f^{-1}( S)\subset Q\setminus (W_1^-\cup W_1^+)$ and $f$ defines a homeomorphism between $\bar f^{-1}(S)$ and $S$.
Since the image $\bar f(C)$ of a continuum $C$ from $W_1^+$ to $W_1^-$ is a continuum from $x$ to $y$ and $S\cap\bar f(C)\ne\emptyset$, the intersection $$\bar{f}^{-1}(S)\cap C=f^{-1}(S)\cap C=f^{-1}(S\cap f(C))=f^{-1}(S\cap\bar f(C))$$ is not empty.
Hence by Corollary~\ref{cor}, $\bar{f}^{-1}(S)$ is strongly infinite dimensional, and so is $S$. \qed
\end{proof}

Clearly, some strongly infinite dimensional compacta can separate the Hilbert cube. Thus, $Q\times\{0\}$ separate the Hilbert cube $Q\times[-1,1]$.
It could be that every strongly infinite dimensional compactum has this property. In other words, it is unclear if the converse to  Theorem~\ref{main1} holds true:
\begin{prob}
Does every strongly infinite dimensional compact metric space admit an embedding into the Hilbert cube $Q$ that separates
$Q$?
\end{prob}

\begin{rem}\label{rem} The preceding theorem in the case when $S$ is compact states precisely that if S is a weakly infinite dimensional subspace of $Q$, then $Q\setminus S$ is path connected. 
\end{rem}

\begin{definition}
A topological space $X$ has property $\mathcal{C}$ (is a  $\mathcal{C}$-space) if for every sequence $\mathcal{G}_1,\mathcal{G}_2,\dots$ of open covers of $X$, there exists a sequence $\mathcal{H}_1,\mathcal{H}_2,\dots$ of families of pairwise disjoint open subsets of $X$ such that for $i=1,2,\dots$ each member of $\mathcal{H}_i$ is contained in a member of $\mathcal{G}_i$ and the union $\bigcup_{i=1}^{\infty}\mathcal{H}_i$ is a cover of $X$. 
\end{definition}

The following is a theorem on dimension lowering mappings, the proof can be found in ~\cite{En} (Chapter 6.3, Theorem 9).

\begin{theorem}\label{cspace}
If $f:X\to Y$ is a closed mapping of space X to $\mathcal{C}$ space $Y$ such that for every $y\in Y$ the fibre $f^{-1}(y)$ is weakly infinite dimensional, then $X$ is weakly infinite dimensional.
\end{theorem}

If one uses weakly infinite dimensional spaces instead of $\mathcal{C}$ spaces the situation is less clear even in the case of
compact $Y$.
\begin{prob}\label{problem}
Suppose that a Lie group $G$ admits a free action by isometries on a metric space $X$ with compact metric 
weakly infinite dimensional orbit space
$X/G$. Does it follow that $X$ is weakly infinite dimensional?
\end{prob}
This is true for compact Lie groups in view of the slice theorem~\cite{Br}. It also true
for countable discrete groups~\cite{Pol}.

\section{Applications}

Now we proceed to generalize two theorems by Belegradek and Hu.  We use the following result proven in~\cite{BH}, Theorem 1.3.

\begin{theorem}\label{ref}
If K is a countable subset of $\mathcal{R}^{k}_{\ge 0}(\mathbb{R}^2)$ and $X$ is a separable metric space, then for any distinct points $x_1,x_2\in X$ and any distinct metrics $g_1,g_2\in \mathcal{R}^{k}_{\ge 0}(\mathbb{R}^2)\setminus K$ there is an embedding of $X$ into $\mathcal{R}^{k}_{\ge 0}(\mathbb{R}^2)\setminus K$ that takes $x_1, x_2$ to $g_1,g_2$ respectively.
\end{theorem}

Here is our extension of the first Belegradek-Hu theorem.
\begin{theorem}\label{appl1}
The complement of every weakly infinite dimensional subspace $S$ of $\mathcal{R}^{k}_{\ge 0}(\mathbb{R}^2)$ is continuum connected. If $S$ is closed, $\mathcal{R}^{k}_{\ge 0}(\mathbb{R}^2)\setminus S$ is path connected.
\end{theorem}

\begin{proof}
Let $S$ be a weakly infinite dimensional subspace of $\mathcal{R}^{k}_{\ge 0}(\mathbb{R}^2)$. Fix two metrics $g_1, g_2 \in \mathcal{R}^{k}_{\ge 0}(\mathbb{R}^2)$. Theorem~\ref{ref} implies that $g_1, g_2$ lies in a subspace $\hat{Q}$ of $\mathcal{R}^{k}_{\ge 0}(\mathbb{R}^2)$ that is homeomorphic to $Q$. 
Since $S\cap \hat{Q}$ is at most weakly infinite dimensional, $\hat{Q}\setminus S$ is continuum connected by Theorem~\ref{main1}. Then $g_1, g_2$ lie in a continuum in $\hat{Q}$ that is disjoint from S. Hence $\mathcal{R}^{k}_{\ge 0}(\mathbb{R}^2)\setminus S$ is continuum connected. 
If $S$ is closed, from Remark~\ref{rem}, we can conclude that  $\mathcal{R}^{k}_{\ge 0}(\mathbb{R}^2)\setminus S$ is path connected.\qed
\end{proof}

In view of Theorem~\ref{main1} we can state that if a subset $S\subseteq \ell^2$, the separable Hilbert space, then $S$ is strongly infinite dimensional. From this fact we derive the following
\begin{theorem}
The complement of every weakly infinite dimensional subspace $S$ of $\mathcal{R}^{\infty}_{\ge 0}(\mathbb{R}^2)$ is locally connected. If $S$ is closed, $\mathcal{R}^{\infty}_{\ge 0}(\mathbb{R}^2)\setminus S$ is locally path connected.
\end{theorem}
\begin{proof}
Given $C^{\infty}$ topology, the space $\mathcal{R}^{\infty}_{\ge 0}(\mathbb{R}^2)$ is homeomorphic to $\ell^2$, the separable Hilbert space ~\cite{BH}. Let $x\in\ell^2$, then there is a neighborhood $U$ of $x$ homeomorphic to $\ell^2$, and the set $U\setminus S$ is connected, and path connected if $S$ is closed.\qed 
\end{proof}
We do not know if the space $\mathcal{R}^{k,c}_{\ge 0}(\mathbb{R}^2)$ locally path connected for $k<\infty$.

We prove a similar to Theorem~\ref{appl1} result for the associated moduli space $\mathcal{M}^{k,c}_{\ge 0}(\mathbb{R}^2)$, when the subspace removed is a Hausdorff space having the property $\mathcal{C}$. 
This is  a generalization of another  Belegradek-Hu theorem.

\begin{theorem}\label{appl2} 
If $S\subset \mathcal{M}^{k,c}_{\ge 0}(\mathbb{R}^2)$ is a closed Hausdorff space with property $\mathcal{C}$   then $\mathcal{M}^{k,c}_{\ge 0}(\mathbb{R}^2)\setminus S$ is path connected.
\end{theorem}
\begin{proof} Denote by $S_1$ the set of smooth subharmonic functions with $\alpha(u)\le 1$ where $$\alpha(u)=\lim_{r\rightarrow \infty}\frac{\sup\{u(z):|z|=r\}}{\log r}.$$
Note that $S_1$ is closed in the Frech\'{e}t space $C^{\infty}(\mathbb{R}^2)$, it is not locally compact, and is equal to the set of smooth subharmonic functions $u$ such that the metric $e^{-2u}g_0$ is complete \cite{BH}. Let $q:S_1\rightarrow \mathcal{M}^{k,c}_{\ge 0}(\mathbb{R}^2)$ denote the continuous surjection sending $u$ to the isometry class of $e^{-2u}g_0$. Let $\hat{S}=q^{-1}(S)$. Fix two points $g_1, g_2\in \mathcal{M}^{k,c}_{\ge 0}(\mathbb{R}^2)\setminus S$. which are $q$ images of $u_1, u_2$ in $S$, respectively. By Theorem ~\ref{ref} we may assume that $u_1, u_2$ lie in an embedded copy $\hat{Q}$ of Hilbert cube. It suffices to show that $ \hat{Q}\setminus \hat{S}$ is path connected.

The set $\hat{Q}\cap\hat{S}$ is compact, hence $\hat{q}$, the restriction of $q$ to $\hat{Q}\cap\hat{S}$ is a continuous surjection. The map $\hat{q}:\hat{Q}\cap\hat{S}\rightarrow q(\hat{Q})\cap S$ is a map between compact spaces, and in particular, it is a closed map. The set $\hat{q}(\hat{Q}\cap\hat{S})$ has property $\mathcal{C}$. We have each fiber $q^{-1}(y)$ to be finite dimensional~\cite{BH}, and hence $\hat{Q}\cap\hat{S}$ is weakly infinite dimensional by Theorem~\ref{cspace}. Therefore, $\hat{Q}\setminus\hat{S}$ is path connected by the discussion following Theorem~\ref{main1}, and so is $ \mathcal{M}^{k}_{\ge 0}(\mathbb{R}^2)\setminus S$ .
\qed
\end{proof}

It should be noted that the Hausdorff condition is essential in Theorem~\ref{appl2}.
If $S$ is not Hausdorff, the map $\hat q$ above ceases to be a map between compact metric spaces. In the paper ~\cite{BH} the authors omitted the Hausdorff condition in the formulation of their
Theorem 1.6 but they use it implicitly in the proof. 

\begin{proposition}
Suppose that Problem~\ref{problem} has an affirmative answer for the Lie group $G=\textrm{conf}(\mathbb{R}^2)$. Then in Theorem~\ref{appl2} one can replace the property $\mathcal C$ condition to the weak infinite dimensionality of $S$.
\end{proposition}
\begin{proof}

We use the same setting as in the proof of theorem~\ref{appl2}. As stated in the proof of theorem 1.6 in~\cite{BH}, two functions $u$ and $v$ of $S_1$ lie in the same isometry class if and only if $v=u\circ\psi -\log|a|$ for some $\psi\in \textrm{conf}(\mathbb{R}^2)$. i.e, they lie in the same orbit under the action of $\textrm{conf}(\mathbb{R}^2)$ on the space $C^{\infty}(\mathbb{R}^2)$ given by $(u,\psi)\mapsto u\circ \psi -\log|a|$. The subspace $S_1$ of $C^{\infty}(\mathbb{R}^2)$ is invariant under this action. Let $\pi:S_1\to S_1/\textrm{conf}(\mathbb{R}^2)$ be the projection onto the orbit space of this action. Also we note that the action of $\textrm{conf}(\mathbb{R}^2)$ on $S_1$ is a free action by isometries.\\
Let $S$ be a closed, weakly infinite dimensional Hausdorff subset of $\mathcal{M}^k_{\ge0}(\mathbb{R}^2)$. Let $f$ and $g$ be two elements in the complement of $S$. Then there are functions $u$ and $v$ mapping to $f$ and  $g$ respectively by $q$. As noted above, $u$ and $v$ lend in the same class if and only if $v=u\circ \psi -\log|a|$ for some $\psi\in \textrm{conf}(\mathbb{R}^2)$. Theorem 1.4 of~\cite{BH} shows that $u$ and $v$ lie in an embedded copy $Q$ of Hilbert cube. Denote $q^{-1}(S)=\hat{S}$. It suffices to prove that $Q\cap\hat{S}$ is weakly infinite dimensional, so we would have a path joining $u$ and $v$ in $Q\setminus\hat{S}$, which transforms to a path joining $g$ and $f$ in $\mathcal{M}^k_{\ge0}(\mathbb{R}^2)\setminus S$. \\
The set $\hat{S}$ is closed hence $Q\cap\hat{S}$ is compact. So the restriction of $q$ to $Q\cap\hat{S}$ is a continuous surjection $\hat{q}: Q\cap\hat{S}\to q(Q)\cap S$ of compact separable metric spaces.
Define the map $\eta:S_1/\textrm{conf}(\mathbb{R}^2)$ by $uG\mapsto u^*$, the isometry class of $e^{-2u}g_0$. This map is injective by definition and the diagram 

 \[
    \xymatrix{ S_1 \ar[rr]^{\pi}\ar[dr]_{q}& &S_1/\textrm{conf}(\mathbb{R}^2)\ar[dl]^{\eta}�\\
    & \mathcal{M}_{\ge0}^{k} &
    }
    \]
 commutes. Let $Y$ be the $\eta$ preimage of $q(Q)\cap S$ in $S_1/\textrm{conf}(\mathbb{R}^2)$. The action restricted to the preimage $\pi^{-1}(Y)$ of $S_1$ is an action of $\textrm{conf}(\mathbb{R}^2)$ on $\pi^{-1}(Y)$ with orbit space $Y$, and $Q\cap\hat{S}\subseteq\pi^{-1}(Y)$, and the set $Y$ is weakly infinite dimensional. Assuming that the Problem~\ref{problem} has an affirmative answer for the Lie group $G=\mathbb{C}^*\rtimes \mathbb{C}$, we can say that Problem~\ref{problem} has an affirmative answer for the Lie group conf$(\mathbb{R}^2)$. By this, we can conclude that $\pi^{-1}(Y)$ is weakly infinite dimensional. Hence $q(Q)\cap\hat{S}$ is weakly infinite dimensional, therefore  $\mathcal{M}^k_{\ge0}(\mathbb{R}^2)\setminus S$ is path connected.
\qed
\end{proof}
\section{Acknowledgements}

I am grateful to my advisor, Alexander Dranishnikov, for all his supports and insights given throughout in preparation of this paper.

\end{document}